\documentclass[a4paper,12pt]{scrartcl}
\usepackage{amsmath, amsfonts, amsthm, amssymb, graphicx, color, hyperref}
\usepackage[right]{eurosym}
\usepackage{tikz}
\usepackage{url}

\newcommand{\cA}{\mathcal{A}}

\newcommand{\cF}{\mathcal{F}}

\newcommand{\cM}{\mathcal{M}}

\newcommand{\IN}{\mathbb{N}}

\newcommand{\IR}{\mathbb{R}}

\newcommand{\R}{\mathbb{R}}

\newcommand{\eps}{\epsilon}

\newcommand{\be}{\begin{eqnarray*}}
\newcommand{\ee}{\end{eqnarray*}}
\newcommand{\ben}{\begin{eqnarray}}
\newcommand{\een}{\end{eqnarray}}

\newtheorem{theo}{Theorem}[section]

\newtheorem{lemma}[theo]{Lemma}
\newtheorem{propo}[theo]{Proposition}
\newtheorem{corollary}[theo]{Corollary}
\newtheorem{ex}[theo]{Example}
\theoremstyle{definition}
\newtheorem{defi}[theo]{Definition}

\newtheorem{remark}[theo]{Remark}

\title{BSDEs with singular terminal condition and control problems with constraints}

\author{Stefan Ankirchner \and Monique Jeanblanc
\and Thomas Kruse \thanks{Ankirchner,
Institute for Applied Mathematics, University of Bonn, Endenicher Allee 60, 53113 Bonn,
Germany. \emph{Email:} \url{ankirchner@hcm.uni-bonn.de}, \emph{Phone:} +49 (0)228 73 3796;
Jeanblanc, Universit\'e d'Evry Val D'Essonne, 23 Boulevard de France, 91025 Evry Cedex, France.
\emph{Email:} \url{monique.jeanblanc@univ-evry.fr}, \emph{Phone:} +33 (0) 1 64 85 34 88/79;
 Kruse, Institute for Applied Mathematics, University of Bonn, Endenicher Allee 60, 53113 Bonn,
  Germany. \emph{Email:} \url{tkruse@uni-bonn.de}, \emph{Phone:} +49 (0)228 73 62 273.
  We are grateful to Xin Guo, Shige Peng and Philipp Strack for helpful comments. Financial support by the German
  Research Foundation (DFG), through the \textit{Hausdorff Center for Mathematics}, and the French Banking federation, through the {\it Chaire Risque de Cr\'edit},  are gratefully
  acknowledged. }}

\begin{document}

\maketitle

\begin{abstract}
We provide a probabilistic solution of a not necessarily Markovian
control problem with a state constraint by means of a Backward
Stochastic Differential Equation (BSDE). The novelty of our
solution approach is that the BSDE possesses a {\it singular
terminal condition}. We prove that a solution of the BSDE
exists, thus partly generalizing existence results obtained
by Popier in \cite{popier2006} and \cite{popier2007}. We perform a
verification and discuss special cases for which the control
problem has explicit solutions.
\end{abstract}
\begin{center}\footnotesize
  \begin{tabular}{r@{ : }p{10cm}}
      {\it 2010 MSC} & 60H10, 91G80, 93E20.\\
      {\it Keywords} &  Stochastic control with constraints, Backward Stochastic differential equations, Maximum principle, Optimal liquidation.
    \end{tabular}
  \end{center}
\section*{Introduction}

In these notes we provide a pure probabilistic solution of the control problem that consists in minimizing the functional
\ben\label{introProb}
J(x)=E\left[\int_0^T (\eta_t\lvert \dot x_t \rvert^p +\gamma_t|x_t|^p) dt\right]
\een
over all absolutely continuous paths $(x_t)_{t \in [0,T]}$ starting in $\xi \in \R$ and
ending in $0$ at time $T$. Here $p>1$ and $(\eta, \gamma)$ are two non-negative stochastic processes that are progressively measurable with respect to the natural filtration $(\cF_t)$ generated by a Brownian motion. We choose the control strategies $x$ to be adapted to $(\cF_t)$.

Such a control problem arises for example when economic agents
have to close a position of $\xi$ asset shares in a market with a stochastic price impact (see e.g.\ \cite{ak11} and the references therein). The first term $\int_0^T \eta_t\lvert
\dot x_t \rvert^p dt$ in \eqref{introProb} can be interpreted
as the liquidity costs entailed by closing the position, where
$\eta$ is a stochastic price impact factor. The second term can be
seen as a measure of the risk associated to the open position.


\bigskip
Our method for solving the control problem \eqref{introProb} draws on the notion of backward stochastic differential equations (BSDEs). BSDEs have turned out to be a powerful tool for analyzing stochastic control problems, and for providing pure probabilistic solutions. We refer to the survey article \cite{epq97} and the book by Pham \cite{pham} for examples of control problems solved with BSDEs.
The control problem \eqref{introProb} considered here imposes a constraint on the terminal value of the control process $x$, namely $x_T = 0$. In the following we characterize its solution with the BSDE
\ben\label{introBSDE}
dY_t=\left((p-1)\frac {Y_t^q}{\eta_t^{q-1}}-\gamma_t\right)
d t+Z_t  d W_t
\een
(where $q = 1/(1 - \frac1p)$) possessing the {\it singular} terminal condition
\ben\label{term condi}
\lim_{t\nearrow T}Y_t=\infty.
\een
We show that if $\eta$ and $\gamma$ satisfy some nice integrability condition, then there exists a minimal solution $(Y,Z)$ of the BSDE \eqref{introBSDE} with terminal condition \eqref{term condi}. We subsequently prove, without any further assumptions, that there exists an optimal control of the problem \eqref{introProb} and that it is given by $x^*_t = \xi e^{-\int_0^t \left(\frac{Y_s}{\eta_s}\right)^{q-1} ds}$. Note that the terminal condition \eqref{term condi} is necessary for the constraint $x^*_T = 0 $ to be satisfied.

One can also derive the singularity \eqref{term condi} by considering the value function associated to the control problem: as $t$ converges to $T$, the value function converges to infinity, provided the position $x \not= 0$. We will show that the value function is a power function of the position variable, multiplied with the solution of the BSDE \eqref{introBSDE}. The singularity of the value function translates into the BSDE's singularity at the terminal condition.

BSDEs with singular terminal conditions have so far been studied only in Popier \cite{popier2006} and \cite{popier2007}. One of the present paper's goal is to reveal their power for
solving the stochastic control problem \eqref{introProb}. BSDEs with singular terminal conditions have not been detected as an efficient tool for solving stochastic control problems yet.

\bigskip
The control problem \eqref{introProb}, more precisely some versions of it, have been already studied in the literature. In \cite{schied2012} a similar class of control problems is solved by means of so-called superprocesses. The functional of the control problem considered in \cite{schied2012} is slightly more general, but the pair $(\eta, \gamma)$ is assumed to be Markovian. The BSDE approach we present here is {\it not} bound to a Markovian model set-up.

Ji and Zhou  \cite{jixyz} consider a very general control problem with terminal state constraints. They assume that the state process is disturbed by some white noise with a volatility that is invertible in the control. Notice that in our setting the state process $x$ is not disturbed.

In \cite{ak11} the authors consider the special case of the control problem \eqref{introProb} where $p =2$, $\eta$ is a constant and $\gamma$ is a function of a homogeneous Brownian martingale (in particular $\gamma$ is a Markov process). They solve the control problem with analytical techniques, characterizing the optimal control and the value function with a solution of a PDE in the viscosity sense.

A probabilistic solution of a related control problem is given in \cite{horstnau} (and the preceding paper \cite{naujokat2011curve}), also by means of BSDEs: the authors consider the problem of how to optimally follow a trading target in an illiquid market with a non-temporary price impact depending on order sizes. Optimal controls, however, are {\it singular} and are verified with BSDEs that have non-singular terminal conditions.

\bigskip
The paper is organized as follows. In Section \ref{probform} we precisely describe the modeling set-up and present the main results. Moreover, we give a heuristic derivation of why the BSDE \eqref{introBSDE} with singular terminal condition provides a solution of the control problem.

In Section \ref{sing_BSDEs} we prove, given some nice integrability conditions on $\eta$ and $\gamma$, that there exists a solution of the BSDE \eqref{introBSDE}.

Section \ref{verification} turns to a verification: we show that the optimal control and value function can indeed be characterized by the BSDE solution constructed in Section \ref{sing_BSDEs}.

Finally, in Section \ref{ind_incr_section} we study in detail the special case where $\gamma$ is zero and $\eta$ has uncorrelated multiplicative increments. We show that in this case the optimal control is deterministic.

\section{Main results}\label{probform}
We fix a deterministic, finite time horizon $T>0$ and a
probability space $(\Omega,\mathcal F,P)$ which supports a
$d$-dimensional Brownian motion $(W_t)_{0\le t\le T}$, where $d
\in \IN$. Let $(\mathcal F_t)_{t\in[0,T]}$ denote the completed
filtration generated $(W_t)_{0\le t\le T}$. Throughout we assume
that $(\eta_t)_{t \in [0,T]}$ and $(\gamma_t)_{t \in [0,T]}$ are
nonnegative, progressively measurable stochastic processes.
We assume $p>1$ and denote by $q = 1/(1 - \frac1p)$ its H\"older
conjugate. We consider the stochastic control problem to minimize
the functional \ben\label{problem} J(x)=E\left[\int_0^T
(\eta_t\lvert \dot x_t \rvert^p +\gamma_t|x_t|^p) dt\right] \een
over all progressively measurable processes
$x:\Omega\times[0,T]\to \IR$ that possess absolutely continuous
sample paths and satisfy the constraints $x_0=\xi \in \R$ and
$x_T=0$ a.s. We denote the set of all these controls by $\mathcal
A_0$, and define \ben\label{min_problem} v=\inf_{x\in \mathcal
A_0}J(x). \een We show that under some nice integrability
conditions on $\eta$ and $\gamma$ there exists an optimal control
$x^* \in \cA_0$; i.e. $J(x^*) = v$. Moreover we characterize the
optimal control by means of a BSDEs with a singular terminal
condition. We define the notion of a solution in the style of
\cite{popier2006}.
\begin{defi}\label{def_sing_BSDE}
 We say that a pair of progressively measurable processes $(Y,Z)$ with values in $\R \times \R^d$ solves the BSDE \eqref{introBSDE} with singular terminal condition $Y_T=\infty$ if it satisfies
\begin{itemize}
 \item [(i)] for all $0\le s\le t<T$: $Y_s=Y_t-\int_s^t \left((p-1)\frac {Y^q_r}{\eta_r^{q-1}}-
 \gamma _r \right) dr -\int_s^tZ_rdW_r$;
\item [(ii)] for all $0\le t<T$: $E\left[\sup_{0\le s \le t}|Y_s|^2+\int_0^t|Z_r|^2dr\right]<\infty$;
\item [(iii)] $\liminf_{t\nearrow T}Y_t=\infty$, a.s.
\end{itemize}
\end{defi}
We introduce the following spaces of processes. For $i=1,2$ and $t \le T$ let
\be
\mathcal M^i(0,t)=L^i(\Omega\times[0,t],\mathcal P,P\otimes \lambda)
\ee
where $\lambda$ is the Lebesgue measure and $\mathcal P$ denotes the $\sigma$-algebra of $(\mathcal F_t)$-progressively measurable subsets of $\Omega\times [0,T]$. Throughout we assume that $\eta$ and $\gamma$ satisfy the integrability conditions
\be
{\bf (I1) } & \phantom{bla bla bla } & \eta \in\mathcal M^2(0,T) \text{ and } 1/\eta^{q-1}\in \mathcal M^1(0,T), \\
{\bf (I2) } & & E\int_0^T (T-s)^p \gamma_s ds < \infty \text { and } \gamma \in \mathcal M^2(0,t) \text{ for all } t<T.
\ee
In our first main result we prove existence of a minimal solution of the BSDE \eqref{introBSDE}.
\begin{theo}\label{first main}
Assume that Conditions $\bf{(I1)}$ and $\bf{(I2)}$ are satisfied. Then there exists a minimal solution  $(Y,Z)$ of the BSDE \eqref{introBSDE} with singular terminal condition $Y_T=\infty$.
\end{theo}
In the second main result we characterize the value function and the optimal control in terms of the minimal solution. \begin{theo}\label{sec main}
Suppose Conditions $\bf{(I1)}$ and $\bf{(I2)}$, and let $(Y,Z)$ be the minimal solution of \eqref{introBSDE}. Then
\be
v = Y_0 |\xi|^p
\ee
and the optimal control is given by
\be
x^*_t=\exp\left(-\int_0^t\left(\frac{Y_s}{\eta_s}\right)^{q-1} ds\right),
\ee
for all $t \in [0,T]$.
\end{theo}
The following deterministic example illustrates that a violation of the integrability condition $1/\eta^{q-1}\in \mathcal M^1(0,T)$ may lead to a minimization problem where no optimal control exists.
\begin{ex}
 Let $T=1$, $\eta_t=(1-t)^\beta$ for some $\beta\ge0$ , $\gamma_t = 0$ and $p=q=2$. Then we have $1/\eta^{q-1} \in L^1([0,T])$ if and only if $\beta<1$. In this case Theorem \ref{sec main} yields that $x_t=(1-t)^{1-\beta}$ is an optimal control. In the case $\beta\ge 1$ consider the control $x_t=(1-t)^\alpha$ for some $\alpha>0$. We compute
\be
J(x)=\int_0^1\eta_t\dot x_t^2dt=\alpha^2\int_0^1(1-t)^{2\alpha +\beta-2}dt.
\ee
Since $\beta\ge 1>1-2\alpha$ the integral is finite and has the value
\be
J(x)=\frac{\alpha^2}{2\alpha+\beta-1}.
\ee
Taking the limit $\alpha\searrow 0$ yields $v=0$, but there exists no control in $\mathcal A_0$ attaining this value.
\end{ex}
\begin{remark}
If $p=1$, then the control problem also does not possess an optimal control in $\mathcal A_0$ (except for some simple cases). For $p=1$ the right formulation of the problem would be to allow for singular controls; and consequently the description of optimal controls would require different methods.
\end{remark}

We prove Theorem \ref{first main} in Section \ref{sing_BSDEs} (see Theorem \ref{existence_singular}) and Theorem \ref{sec main} in Section \ref{verification} (see Theorem \ref{opt_strat_constr}). Before tackling the proofs we provide a heuristic derivation of the BSDE \eqref{introBSDE}.
\subsection*{Heuristic derivation of the BSDE}
Throughout this section we assume $\xi>0$. First we show that
in this case we can restrict attention to non-increasing non-negative controls.
To this end we denote the set of controls in $\mathcal A_0$ with
non-increasing sample paths by $\mathcal D _0$.

\begin{lemma}\label{nonincr_strategies}
Every control $x\in \mathcal A_0$ can be modified to a control $\underline x \in \mathcal D_0$ such that $J(x)\ge J(\underline x)$. In particular, we have $v=\inf_{x\in \mathcal D_0}J(x)$.
\end{lemma}
\begin{proof}
Let $x\in \mathcal A_0$ and define its running minimum cut off at
zero by $\underline x_t=\min_{0\le s\le t}x_s\vee 0$. Notice that
$\underline x$ is absolutely continuous
since $\underline x_t = \int_0^t \dot x_s 1_{\{ x_s = \underline
x_s\}} ds$. Hence $\underline x \in \mathcal D_0$. Observe that $|\dot {\underline x}_t| \le |\dot x_t|$, and therefore we have
$E\left[\int_0^T\eta_t|\dot x_t|^pdt\right]\ge
E\left[\int_0^T\eta_t|\dot {\underline x}_t|^pdt\right]$. Since
$\underline x \le x$ on $\Omega \times [0,T]$ it follows that
$E\left[\int_0^T\gamma_t|x_t|^pdt\right]\ge
E\left[\int_0^T\gamma_t|{\underline x}_t|^pdt\right]$. Thus, we
obtain $J(x)\ge J(\underline x)$.
\end{proof}
The next result, a maximum principle, provides a sufficient condition for optimality in \eqref{min_problem}. We remark that we use it only for the heuristic derivation of the BSDE \eqref{introBSDE}. The rigorous verification in Section \ref{verification} will be performed via a penalization.
\begin{propo}[Maximum principle] \label{suff_cond}
 Assume that $x\in \mathcal D_0$ and that $M_t=p\eta_t |\dot x_t|^{p-1}+p\int_0^t\gamma_sx_s^{p-1}ds$ is a martingale with $E[M_T^2]<\infty$. Then $x$ is optimal in \eqref{min_problem}.
\end{propo}
\begin{proof}
 Let $g(z)=\lvert z \rvert ^p$ and $x\in \mathcal D_0$ such that $M_t=p\eta_t |\dot x_t|^{p-1}+p\int_0^t\gamma_sx_s^{p-1}ds$ is a martingale with $E[M_T^2]<\infty$. Let $ y \in \mathcal D_0$ and introduce $\theta_t=x_t- y_t$. Then $\theta$ satisfies $\theta_0=\theta_T=0$ a.s. Furthermore, since $x$ and $y$ are non-increasing it follows that $\theta$ is bounded: $|\theta_t|\le 2 |\xi|$. Since $\dot x\le 0$ on $\Omega \times [0,T]$ we have $g'(\dot x_t)=-p|\dot x_t|^{p-1}$. The convexity of $g$ implies for all $t\in[0,T]$
\be
g(\dot x_t)-g(\dot y_t)\le g'(\dot x_t)(\dot x_t-\dot y_t).
\ee
Thus, by integration by parts we obtain
\be
\int_0^T\eta_t(g(\dot x_t)-g(\dot y_t))dt&\le& \int_0^T\eta_tg'(\dot x_t)d\theta_t=\int_0^T\left(\int_0^tp\gamma_sx^{p-1}_sds-M_t \right)d\theta_t\\
&=&\int_0^T\theta_tdM_t-\int_0^T\gamma_tg'(x_t)\theta_tdt.
\ee
Since $\theta$ is bounded and $M$ is a martingale with $E[M_T^2]<\infty$ it follows that the integral process $\int_0^\cdot \theta_tdM_t$ is a martingale starting in $0$. In particular, it vanishes in expectation. Using again the convexity of $g$ yields for $t\in[0,T]$
\be
g(x_t)-g(y_t)\le g'(x_t)( x_t-y_t).
\ee
Taking expectations implies optimality of $x$:
\be
E\left[\int_0^T\eta_t(g(\dot x_t)-g(\dot y_t))dt\right]\le -E\left[\int_0^T\gamma_tg'(x_t)\theta_tdt\right]\le -E\left[\int_0^T\gamma_t(g(x_t)-g(y_t))dt\right].
\ee
\end{proof}
\begin{remark}
 Observe that the previous two results hold in a more general framework than the one under consideration here. We can replace $y\mapsto |y|^p$ by any convex function which attains its minimum at $y=0$.
\end{remark}

We next observe that the {\it relative} control rate $r_t=\frac{\dot x_t}{x_t}$ of an optimal control $x \in \mathcal A_0$ is independent of the current state $x_t$. To this end fix $t<T$ and $\xi_2>\xi_1>0$. Assume that $(x^1_s)_{t\le s\le T}$ is an optimal control to close the position $\xi_1$ in the period $[t,T]$. Then the homogeneity of $y\mapsto |y|^p$ implies that the control $x^2_s=\frac{\xi_2}{\xi_1}x^1_s$, $s\in[t,T]$, is optimal to close the position $\xi_2$ in the period $[t,T]$. In particular the relative control rates at time $t$ coincide $\frac{\dot x^1_t}{\xi_1}=\frac{\dot x^2_t}{\xi_2}$. Hence, an optimal control can be represented in feedback form $\dot x_t=r_tx_t$, where $r_t$ is the relative control rate, which does not depend on $x_t$. We denote by $q$ the H\"older conjugate of $p$ and rewrite $r_t$ as $r_t=-\left(\frac {Y_t}{\eta_t}\right)^{q-1}$ for some semi-martingale $Y$ and make the ansatz that an optimal control $x$ is of the form
\ben\label{ode_of_x}
\dot x_t=-\left(\frac {Y_t}{\eta_t}\right)^{q-1}x_t
\een
 with $x_0=1$. The solution of this pathwise ordinary differential equation is given by
\ben\label{expl_rep_x}
x_t=e^{-\int_0^t\left(\frac {Y_s}{\eta_s}\right)^{q-1}ds}.
\een

Proposition \ref{suff_cond} shows that $x\in \mathcal A_0$ is optimal if the process
$p\eta |\dot x |^{p-1}+p\int_0^\cdot \gamma_sx_s^{p-1}ds$ is a martingale. Since $(\mathcal F_t)$ is a Brownian filtration this is
equivalent to the existence of a predictable process $\phi$ such
that
\be
d (\eta |\dot x| ^{p-1})_t+\gamma_tx_t^{p-1}dt=\phi_t d W_t.
\ee
Using the equality $\eta_t|\dot x_t|^{p-1}=Y_tx_t^{p-1}$ and applying the integration by parts formula to the product $Y x^{p-1}$ we obtain
\be
 d (\eta |\dot x|^{p-1})_t+\gamma_tx_t^{p-1}dt&=&x_t^{p-1}  d Y_t +(p-1)Y_tx_t^{p-2} d x_t +\gamma_tx_t^{p-1}dt\\
&=&x_t^{p-1}\left(  d Y_t -\left((p-1)\frac {Y_t^q}{\eta_t^{q-1}}-\gamma_t\right) d t\right).
\ee
Setting $Z_t=\phi_t/x_t^{p-1}$ we see that $Y$ satisfies the BSDE
\ben\label{BSDE1}
dY_t=\left((p-1)\frac {Y^q_t}{\eta_t^{q-1}} -\gamma_t\right)d t+Z_t  dW_t.
\een
In view of Equation \eqref{expl_rep_x} the singular terminal condition $Y_T=\infty$ is necessary to ensure that $x_T=0$. In Theorem \ref{opt_strat_constr} we show that this condition is indeed sufficient.

\section{Construction of a BSDE solution with singular terminal condition}\label{sing_BSDEs}
In this section we construct a solution of the BSDE \eqref{introBSDE} with singular terminal condition. To this end we first show existence of solutions to BSDEs with cut off drivers and finite deterministic terminal condition $L>0$. In a second step we let $L$ tend to infinity and obtain a solution with a singular terminal condition. We show that this particular solution is the minimal solution of \eqref{BSDE1}. We remark that the second step of our construction bears similarities with the existence proof conducted by Popier in \cite{popier2006} resp.\ \cite{popier2007}.

Let us clarify some terminology concerning BSDEs. The pair consisting of the driver and the terminal condition of a BSDE will be referred to as its {\it parameters}. Given a solution $(Y,Z)$ of a BSDE, we call the first component $Y$ the {\it solution process} and the second component $Z$ the {\it martingale component}. 

\subsection{Approximation}
Consider the BSDE
\ben\label{finite_BSDE}
dY^L_t=\left((p-1)\frac {(Y^L_t)^q}{\eta_t^{q-1}} - (\gamma_t \wedge L)\right) d t+Z^L_t  d W_t,
\een
with terminal condition $Y^L_T=L$.
\begin{propo}\label{finite_existence}
 Assume that $\eta \in \mathcal M^2(0,T)$ and $\frac 1{\eta^{q-1}} \in \mathcal M^1(0,T)$. Then there exists a solution $(Y^L,Z^L)$ to \eqref{finite_BSDE} with $Z^L\in \mathcal M^2(0,T)$. For every $t\in[0,T]$ the random variable $Y_t^L$ is bounded from below and above as follows
\ben\label{bounds_YL}
\frac 1 {\left( \frac1{L^{q-1}}+E\left[\int_t^T\frac{1}{\eta_s^{q-1}} ds\big|\mathcal F_t\right] \right)^{p-1}}\le Y^L_t\le (1+T)L\wedge \frac 1{(T-t)^p}E\left[\left.\int_t^T (\eta_s + (T-s)^p \gamma_s) ds \right| \mathcal F_t\right].
\een

\end{propo}
\begin{proof}
 Let $f(t,y)=-(p-1)\frac {y^q}{\eta_t^{q-1}} + (\gamma_t \wedge L)$ denote the driver of the BSDE \eqref{finite_BSDE}.
 Define $f^\delta(t,y)=-(p-1)\frac {y^q}{(\eta_t\vee\delta)^{q-1}} + (\gamma_t \wedge L)$ for $\delta>0$. Being decreasing in $y$, bounded in $\omega$, the driver $(\omega,t,y)\mapsto f^\delta(t,y\vee 0)$ - which does not depend on $z$ - satisfies all conditions of Theorem 2.2. in \cite{pardoux1999bsdes}. Hence, for every $L>0$ there exists a solution $(Y^{\delta,L},Z^{\delta,L})$ to the BSDE with parameters $(f^\delta(t,y\vee 0),L)$. Moreover, any such solution satisfies
\ben\label{int_estimate_Ydelta}
E\left[\sup_{0\le t\le T}|Y^{\delta,L}_t|^2+\int_0^T (Z_t^{\delta,L})^2 d
t\right]<\infty.
\een
For $L=0$ the solution is given by
$(Y^{\delta,0},Z^{\delta,0})=(0,0)$. The comparison theorem
\cite[Theorem 2.4]{pardoux1999bsdes} implies that $Y^{\delta,L}$
is nonnegative and, hence, $Y^{\delta,L}$ is also a solution to
the BSDE with parameters $(f^\delta,L)$.

We can also derive an upper bound for $Y^{\delta,L}$ by appealing to the comparison theorem. Note that we have $f^\delta(t,y)\le L$ for $y\ge 0$. This implies
\ben\label{up_bound_Ydelta}
Y^{\delta,L}_t\le (1+T) L
\een
for all $t\in[0,T]$.

We obtain a solution of the BSDE \eqref{finite_BSDE} by letting $\delta$ converge to zero. Indeed, the mapping
$\delta\mapsto f^\delta$ is increasing, which implies that
$Y^{\delta_1,L}\le Y^{\delta_2,L}$ if $\delta_1\le \delta_2$. In
particular we can define $Y^L$ as the decreasing limit of
$Y^{\delta,L}$ as $\delta \searrow 0$. For the convergence of
the control process $Z^{\delta,L}$, let $(\delta_n)_{n\ge 0}$ be a
sequence with $\delta_n\searrow 0$ as $n \to \infty$. Fix $n\ge
m$. Then we have $Y^{\delta_n,L}\le Y^{\delta_m,L}$. For all $0\le
t\le T$ It\^o's formula leads to
\begin{equation}\label{control_estimate}
 \begin{split}
  \int_0^T(Z_s^{\delta_n,L}-Z_s^{\delta_m,L})^2 d s =& -(Y_0^{\delta_n,L}-Y_0^{\delta_m,L})^2 \\
&-2\int_0^T (Y_s^{\delta_n,L}-Y_s^{\delta_m,L})(Z_s^{\delta_n,L}-Z_s^{\delta_m,L}) d W_s \\
&+2 \int_0^T(Y_s^{\delta_n,L}-Y_s^{\delta_m,L})(f^{\delta_n}(s,Y_s^{\delta_n,L})-f^{\delta_m}(s,Y_s^{\delta_m,L})) d s
 \end{split}
\end{equation}
Estimates \eqref{int_estimate_Ydelta} and \eqref{up_bound_Ydelta} imply
\be
E\left[\int_0^T (Y_s^{\delta_n,L}-Y_s^{\delta_m,L})(Z_s^{\delta_n,L}-Z_s^{\delta_m,L}) d W_s\right]=0.
\ee
By monotonicity of $f^{\delta_m}$ and estimate \eqref{up_bound_Ydelta} we have
\begin{equation*}
\begin{split}
(Y_s^{\delta_n,L}-Y_s^{\delta_m,L})(f^{\delta_n}&(s,Y_s^{\delta_n,L})-f^{\delta_m}(s,Y_s^{\delta_m,L})) \\
&\le (Y_s^{\delta_n,L}-Y_s^{\delta_m,L})(f^{\delta_n}(s,Y_s^{\delta_n,L})-f^{\delta_m}(s,Y_s^{\delta_n,L})) \\
&=(p-1)(Y_s^{\delta_m,L}-Y_s^{\delta_n,L})(Y_s^{\delta_n,L})^q\left(\frac 1{(\eta_s\vee \delta_n)^{q-1}}-\frac 1{(\eta_s\vee \delta_m)^{q-1}}\right)\\
&\le C \left(\frac 1{(\eta_s\vee \delta_n)^{q-1}}-\frac 1{(\eta_s \vee \delta_m)^{q-1}}\right)
\end{split}
\end{equation*}
for all $s\in [0,T]$ and a constant $C>0$. Taking expectations
in Equation \eqref{control_estimate} yields
\be
E\left[\int_0^T(Z_s^{\delta_n,L}-Z_s^{\delta_m,L})^2 d
s\right]\le 2C E\left[\int_0^T\left(\frac 1{(\eta_s\vee
\delta_n)^{q-1}}-\frac 1{(\eta_s\vee
\delta_m)^{q-1}}\right)ds\right].
 \ee
The sequence $\left(\frac
1{(\eta\vee \delta_n)^{q-1}}\right)_{n\ge0}$ converges in
$\mathcal M^1(0,T)$ to $\frac 1{\eta^{q-1}}$ as $n\to
\infty$. This implies that $(Z^{\delta_n,L})_{n\ge 0}$ is a Cauchy
sequence in $\mathcal M^2(0,T)$ and converges to $Z^L\in
\mathcal M^2(0,T)$. In particular the random variable $\int_t^TZ^{\delta_n,L}_r  d W_r$ converges to $\int_t^TZ^{L}_r  d W_r$ in $L^2(\Omega)$ as $n\to \infty$. We obtain almost sure convergence by passing to a subsequence. Taking the limit $n \to \infty$ in
\be
Y_t^{\delta_n,L}=L-(p-1)\int_t^T\frac
{(Y_r^{\delta_n,L})^q}{(\eta_r\vee\delta_n)^{q-1}} d
r-\int_t^TZ^{\delta_n,L}_r  d W_r, \ee
and using estimate
\eqref{up_bound_Ydelta} yields that $(Y^L,Z^L)$ satisfies almost surely the BSDE
\be
Y^L_t=L-(p-1)\int_t^T\frac {(Y_r^L)^q}{\eta_r^{q-1}} d
r-\int_t^TZ^L_r  d W_r.
\ee

We proceed by deriving the upper and lower bound in \eqref{bounds_YL}. We first estimate $Y^{\delta,L}$ against a linear BSDE with driver
\be
g(t,y)=-p\frac{y}{T-t} + \frac{\eta_t\vee \delta}{(T-t)^p} + \gamma_t.
\ee
By using the inequality
\be
(p-1)y^q-pa^{q-1}y+a^q\ge 0,
\ee
which holds for all $y\ge 0, a\ge 0$, one can show that $f^\delta(t,y)\le g(t,y)$ (take $a=(\eta_t\vee \delta)(T-t)^{-p/q})$. Let $\eps > 0$ and denote by $\Psi^\eps$ the solution process of the BSDE on $[0,T-\eps]$ with parameters $(g, Y^{\delta,L}_{T-\eps})$. By the solution formula for linear BSDEs we have
\be
\Psi^{\eps}_t=E\left[\Gamma_{T-\eps} Y^{\delta,L}_{T-\eps} + \int_t^{T-\eps} \Gamma_s\left(\frac{\eta_s\vee \delta}{(T-s)^p} + \gamma_s\right) ds|\mathcal F_t\right],
\ee
where
\be
\Gamma_t &=& \exp\left(-\int_0^t\frac{p}{T-s} ds \right) = \left( \frac{T-t}{T} \right)^p.
\ee
The comparison theorem implies
\ben\label{up_bound_Ydelta2}
Y_t^{\delta,L}\le \Psi^{\eps}_t= \frac{1}{(T-t)^p} E\left[\eps^p Y^{\delta,L}_{T-\eps} + \int_t^{T-\eps} \left((\eta_s\vee \delta) + (T-s)^p \gamma_s\right) ds|\mathcal F_t\right]
\een
for all $t\in [0,T]$ and $\eps>0$. By letting $\eps \downarrow 0$ we obtain with dominated convergence
\be
Y_t^{\delta,L}\le \frac{1}{(T-t)^p} E\left[\int_t^{T} \left((\eta_s\vee \delta) + (T-s)^p \gamma_s\right) ds|\mathcal F_t\right].
\ee
By letting $\delta \downarrow 0$ we obtain the upper bound in \eqref{bounds_YL}.

In order to derive the lower estimate,
let $V_t = \frac{1}{L^{q-1}} + E\left[\int_t^T\frac{1}{(\eta_s\vee\delta)^{q-1}} d s\big|\mathcal F_t\right]$, and observe that there exists a process $Z \in \cM^2(0,T)$ such that
\be
dV_t = -\frac{1}{(\eta_t \vee\delta)^{q-1}} dt + Z_t dW_t.
\ee
Notice that $\frac{1}{L^{q-1}} \le V_t \le \kappa := \frac{1}{L^{q-1}} + \frac{T}{\delta^{q-1}}$. Next let $U_t = \frac{1}{V_t^{p-1}}$. With Ito's formula one can show that there exists $\tilde Z \in \cM^2(0,T)$ such that
\be
dU_t = -h(t,U_t, \tilde Z_t) dt + \tilde Z_t dW_t,
\ee
where
\be
h(t,u,z) = -(p-1)\frac{(u\wedge L)^q}{(\eta_s\vee\delta)^{q-1}} - \frac12 \frac{p}{p-1} \frac{z^2}{u \vee (1/\kappa^{p-1})}.
\ee
Note that $h(t,u,z) \le f^\delta(t,u)$. Since $U_T = L = Y^{\delta, L}_T$, the comparison theorem for quadratic BSDEs (see e.g. Theorem 2.6 in \cite{kobylanski2000backward}) implies that $U_t \le Y^{\delta,L}_t$. Finally, by letting $\delta \downarrow 0$, we obtain the lower estimate in \eqref{bounds_YL}.

\end{proof}

\subsection{Existence of solutions for BSDEs with singular terminal condition}
First we establish the convergence of $(Y^L,Z^L)$ from Proposition \ref{finite_existence} to a pair $(Y,Z)$ which is a solution to the BSDE \eqref{BSDE1} with singular terminal condition $Y_T=\infty$ in the sense of Definition \ref{def_sing_BSDE}.
\begin{theo}\label{existence_singular}
Assume $\bf{(I1)}$ and $\bf{(I2)}$ hold true. Let $(Y^L,Z^L)$ be the solution to \eqref{finite_BSDE} from Proposition \ref{finite_existence}. Then there exists a process $(Y,Z)$ such that for every $0\le t<T$ the random variable $Y^L_t$ converges a.s.\ to $Y_t$ and $Z^L$ converges in $\mathcal M^2(0,t)$ to $Z$ as $L\to \infty$. The limit process $(Y,Z)$ is a solution to the BSDE \eqref{BSDE1} with singular terminal condition $Y_T=\infty$. Moreover, for every $t\in[0,T]$ the random variable $Y_t$ is almost surely positive:
\ben\label{low_bound_Y}
Y_t\ge \frac 1 {\left( E\left[\int_t^T\frac{1}{\eta_s^{q-1}} ds\big|\mathcal F_t\right] \right)^{p-1}}.
\een
\end{theo}

\begin{proof}
The proof is partly a generalization of the arguments in \cite{popier2006} to our setting. Appealing to the comparison theorem \cite[Theorem 2.4]{pardoux1999bsdes} yields that $Y^L\le Y^N$ if $N>L$ (Observe that although assumption (ii) of \cite[Theorem 2.4]{pardoux1999bsdes} is not satisfied here, the comparison holds, since the process $\alpha_t$ from the proof is non-positive here as well). By Equation \eqref{bounds_YL} for fixed $t<T$ the family of random variables $(Y_t^L,{L\ge 0})$ is bounded from above as follows
\ben\label{up_boundY}
Y_t^L\le \frac 1{(T-t)^p}E\left[\int_t^T (\eta_s  + (T-s)^p \gamma_s) ds|\mathcal F_t\right].
\een
Hence, for all $t<T$ we can define $Y_t$ as the increasing limit of $Y_t^L$ as $L\to \infty$. Notice that by Conditions $\bf{(I1)}$ and $\bf{(I2)}$ the random  variable on the RHS of \eqref{up_boundY} is square integrable. By dominated convergence, therefore, $Y^L_t$ converges to $Y_t$ in $L^2(\Omega)$.

Taking the limit $L\nearrow \infty$ in the lower bound for $Y^L$ from Inequality \eqref{bounds_YL} yields that $Y$ satisfies \eqref{low_bound_Y}. We write $E\left[\int_t^T\frac{1}{\eta_s^{q-1}} ds|\mathcal F_t\right]=M_t-A_t$ with $M_t=E\left[\int_0^T\frac{1}{\eta_s^{q-1}} ds|\mathcal F_t\right]$ and $A_t=\int_0^t\frac{1}{\eta_s^{q-1}} ds$. Since $(\mathcal F_t)$ is a Brownian filtration the martingale $M$ is continuous. This implies
\be
\lim_{t\nearrow T}E\left[\int_t^T\frac{1}{\eta_s^{q-1}} ds\big|\mathcal F_t\right]=\lim_{t\nearrow T} (M_t-A_t) =M_T-A_T=0.
\ee
Hence, it follows from \eqref{low_bound_Y} that $Y$ satisfies the singular terminal condition $\liminf_{t\nearrow T}Y_t=\infty$.

For the convergence of $(Z^L)$ let $0\le s\le t<T$. Then It\^o's formula  implies, for $N, L \ge 0$,
\begin{equation}
 \begin{split}
  (Y_s^N-Y_s^L)^2+\int_s^t |Z^N_r-Z^L_r|^2 d r &= (Y_t^N-Y_t^L)^2-2\int_s^t(Y^N_r-Y^L_r)(Z_r^N-Z_r^L) d W_r \\
& \quad + 2 \int_s^t (Y_r^N-Y_r^L)(f^N(r,Y_r^N)-f^L(r,Y_r^L)) d r. 
 \end{split}
\end{equation}
The monotonicity of the driver $f^L(r,y)=-(p-1)\frac {y^q}{\eta_t^{q-1}} + (\gamma_r \wedge L)$ in $y$ yields for $y, y' \ge 0$
\be
(y-y')(f^N(r,y)-f^L(r,y')) \le (y-y')(f^N(r,y)-f^L(r,y)) = (y-y')(\gamma_r \wedge N - \gamma_r \wedge L),
\ee
and hence
\begin{equation}\label{est_approx}
 \begin{split}
  (Y_s^N-Y_s^L)^2+\int_s^t |Z^N_r-Z^L_r|^2 d r
&\le (Y_t^N-Y_t^L)^2-2\int_s^t(Y^N_r-Y^L_r)(Z_r^N-Z_r^L) d W_r \\
&  \quad + 2 \int_s^t (Y_r^N-Y_r^L)(\gamma_r\wedge N - \gamma_r \wedge L) d r.
 \end{split}
\end{equation}
Since $Y^L$ and $Y^N$ are bounded and $Z^L ,Z^N\in \mathcal M^2(0,T)$, we have
\be
E\left[\int_s^t(Y^N_r-Y^L_r)(Z_r^N-Z_r^L) d W_r\right]=0.
\ee
Then estimate \eqref{est_approx} implies
\ben\label{up_bound_controls}
E\left[\int_0^t |Z_r^N-Z^L_r|^2 dr\right]\le E\left[(Y^N_t-Y^L_t)^2] + 2 \int_s^t (Y_r^N-Y_r^L)(\gamma_r\wedge N - \gamma_r \wedge L) dr\right],
\een
and for a constant $C_1$
\ben\label{esti sup}
E\left[\sup_{0\le s\le t}(Y_s^N-Y^L_s)^2\right]
&\le& E[(Y^N_t-Y^L_t)^2]+C_1E\left[\sqrt{\int_0^t(Y^N_r-Y^L_r)^2 |Z_r^N-Z_r^L|^2 dr}\right] \nonumber \\
& & \quad + 2 E\left[\int_0^t (Y_r^N-Y_r^L)(\gamma_r\wedge N - \gamma_r \wedge L) dr\right],
\een
where we used the Burkholder-Davis-Gundy inequality. From Young's inequality we derive
\be
E\left[\sqrt{\int_0^t(Y^N_r-Y^L_r)^2 |Z_r^N-Z_r^L|^2 dr}\right]&\le& E\left[\sup_{0\le s\le t}|Y_s^N-Y^L_s|\sqrt{\int_0^t |Z_r^N-Z_r^L|^2 dr}\right] \\
&\le& \frac1{4C_1}E\left[\sup_{0\le s\le t}(Y_s^N-Y^L_s)^2\right]+C_1E\left[\int_0^t |Z_r^N-Z^L_r|^2 dr\right],
\ee
which implies, together with \eqref{esti sup} and \eqref{up_bound_controls},
\be
\frac34 E\left[\sup_{0\le s\le t}(Y_s^N-Y^L_s)^2\right] 
&\le &  C_2 E[(Y^N_t-Y^L_t)^2] \\
& & + 2 C_2 E\left[\int_0^t (Y_r^N-Y_r^L)(\gamma_r\wedge N - \gamma_r \wedge L) dr\right],
\ee
where $C_2 = 1 + C_1^2$. Again with Young's inequality we get
\be
& & E\left[\int_0^t (Y_r^N-Y_r^L)(\gamma_r\wedge N - \gamma_r \wedge L) dr\right] \\
&\le& \frac{1}{4 C_2} E\left[\sup_{0\le s\le t}(Y_s^N-Y^L_s)^2\right] + C_2 E\left[\left(\int_0^t |\gamma_r\wedge N - \gamma_r \wedge L| dr \right)^2 \right].
\ee
Finally we arrive at
\ben\label{up_bound_diffYL}
E\left[\sup_{0\le s\le t}(Y_s^N-Y^L_s)^2\right]
&\le &  C_3 E\left[(Y^N_t-Y^L_t)^2 + \int_0^t (\gamma_r\wedge N - \gamma_r \wedge L)^2 dr\right],
\een
for a constant $C_3 \ge 0$.
The RHS of \eqref{up_bound_diffYL} converges to zero as $N$, $L \to \infty$. In particular, Inequality \eqref{up_bound_controls} implies that $(Z^L)$ is a Cauchy sequence in $\mathcal M^2(0,t)$ and converges to $Z \in \mathcal M^2(0,t)$ for every $t<T$. Moreover, Inequality \eqref{up_bound_diffYL} yields that $E\left[\sup_{0\le s\le t}Y_s^2\right]<\infty$. Finally, taking the limit $L\nearrow \infty$ in
\be
Y^L_s=Y^L_t-\int_s^t \left((p-1)\frac {(Y^L_r)^q}{\eta_r^{q-1}} - \gamma_r\right) dr -\int_s^tZ^L_rdW_r
\ee
implies that $Y$ satisfies \eqref{BSDE1} for every $0\le s\le t<T$.
\end{proof}

\begin{propo}\label{minimality}
 The solution obtained in Theorem \ref{existence_singular} is minimal: If $(Y',Z')$ is another nonnegative solution of \eqref{BSDE1} with singular terminal condition $Y_T=\infty$, then $Y'_t\ge Y_t$ a.s.\ for all $t\in[0,T]$.
\end{propo}
\begin{proof}
The proof is an adaptation of \cite[Theorem 7]{popier2007} to our setting.

Fix $L>0$ and let $(Y^L,Z^L)$ denote the solution of \eqref{finite_BSDE} with terminal condition $Y^L_T=L$. Let $(Y',Z')$ be a nonnegative solution of \eqref{BSDE1} in the sense of Definition \ref{def_sing_BSDE}. Set $\Delta_t=Y'_t-Y^L_t$, $\Gamma_t=Z'_t-Z^L_t$ and
\be
\alpha_t=\left\{\begin{aligned} \frac{p-1}{\eta_t^{q-1}}\frac{(Y'_t)^q-(Y^L_t)^q}{Y'_t-Y^L_t}&,& \quad \text{if } \eta_t^{q-1}(Y'_t-Y^L_t)\neq 0 \\
                0&,& \quad \text{else.}
               \end{aligned}\right.
\ee
Note that $\alpha$ is nonnegative. For every $t<T$ the process $(\Delta,\Gamma)$ solves the linear BSDE
\be
d\Delta_s=[\alpha_s\Delta_s - (\gamma_t-L)^+]ds+\Gamma_sdW_s
\ee
on $[0,t]$ with terminal condition $\Delta_t=Y'_t-Y^L_t$. Hence, by Lemma \ref{LBSDE} in the Appendix the solution $\Delta$ admits the explicit representation
\be
\Delta_s=E\left[\Delta_t e^{-\int_s^t\alpha_r dr} + \int_s^t e^{-\int_s^u\alpha_r dr } (\gamma_u-L)^+ du|\mathcal F_s\right].
\ee
Since $Y'$ is nonnegative and $Y^L\le (1+T)L$ by Proposition \ref{finite_existence}, we have $\Delta_t\ge -(1+T)L$. Thus $\Delta_t e^{-\int_s^t\alpha_r dr}$ is bounded from below by $-(1+T)L$ and we can apply Fatou's lemma to obtain
\be
Y'_s-Y^L_s&=&\Delta_s=\liminf_{t\nearrow T} E\left[\Delta_t e^{-\int_s^t\alpha_r dr} + \int_s^t e^{-\int_s^u\alpha_r dr } (\gamma_u-L)^+ du|\mathcal F_s\right] \\
&\ge& E\left[\liminf_{t\nearrow T}\Delta_t e^{-\int_s^t\alpha_r dr}|\mathcal F_s\right]
\ge0.
\ee
Finally, taking the limit $L\nearrow \infty$ yields the claim.
\end{proof}

\section{Optimal Controls}\label{verification}
In this section we first consider a variant of the minimization problem \eqref{min_problem}, where we omit the constraint $x_T=0$ in the set of admissible controls but penalize any nonzero terminal state by $L|x_T|^p$. We show that optimal controls for this unconstrained minimization problem admit a representation in terms of the solutions $Y^L$ from Proposition \ref{finite_existence}. We then use this result to derive an optimal control for \eqref{min_problem}.

Throughout this section we assume $\bf{ (I1)}$ and $\bf{(I2)}$ without further mentioning it.
\subsection{Penalization}\label{penalization}
In this section we consider the unconstrained minimization problem
\ben\label{penprob}
v^L=\inf_{x\in \mathcal A} J^L(x)=\inf_{x\in \mathcal A} E\left[\int_0^T \big(\eta_t\lvert \dot x_t \rvert^p +(\gamma_t\wedge L)|x_t|^p\big)dt+L\lvert x_T \rvert^p\right]
\een
for some $L>0$, where we take the infimum over $\mathcal A$, the set of all progressively measurable processes $x:\Omega \times [0,T]\to \IR$ with absolutely continuous sample paths starting in $x_0=\xi$. Next, we show how to obtain a minimizing control for \eqref{penprob} from the solution $Y^L$ to \eqref{finite_BSDE}.
\begin{propo}\label{penverification}
 Let $(Y^L,Z^L)$ be the solution to \eqref{finite_BSDE} from Proposition \ref{finite_existence}. Then
\be
x^L_t=\xi e^{-\int_0^t\left(\frac {Y^L_s}{\eta_s}\right)^{q-1}ds}
\ee
is optimal in \eqref{penprob} and we have $v^L=Y^L_0|\xi|^p$.
\end{propo}

\begin{proof}
To simplify notation we assume $\xi = 1$ and set $\gamma^L_t=\gamma_t\wedge L$. Let $g(z)=\lvert z \rvert ^p$ and $M_t=pY^L_t(x_t^L)^{p-1}+p\int_0^t\gamma^L_s(x^L_s)^{p-1}dt$.
Applying the integration by parts formula to $M$ results in
\be
dM_t&=&p(x_t^L)^{p-1}dY^L_t+p(p-1)Y^L_t(x_t^L)^{p-2}dx_t^L+p\gamma^L_t(x_t^L)^{p-1}dt\\
&=&p(x_t^L)^{p-1}Z^L_tdW_t.
\ee
Since $x^L$ is bounded and $Z^L\in \mathcal M^2(0,T)$, the process $M$ is a martingale.
Let $x\in \mathcal A$ and introduce $\theta_t=x^L_t-x_t$. Then $\theta$ satisfies $\theta_0=0$. Similar considerations as in Lemma \ref{nonincr_strategies} imply that we can assume that $x$ is pathwise non-increasing and hence $|\theta_t|\le 2$. Furthermore, we have $\eta_tg'(\dot x^L_t)=-p\eta_t|\dot x^L_t|^{p-1}=-pY^L_t(x_t^L)^{p-1}$. The convexity of $g$ implies for all $t\in[0,T]$
\be
g(\dot x^L_t)-g(\dot x_t)\le g'(\dot x^L_t)(\dot x^L_t-\dot x_t).
\ee
Thus, it follows from integration by parts
\be
\int_0^T\eta_t(g(\dot x^L_t)-g(\dot x_t))dt&\le& \int_0^T\eta_tg'(\dot x^L_t)d\theta_t=\int_0^T\left(\int_0^tp\gamma^L_sx^{p-1}_sds-M_t\right)d\theta_t\\
&=&-Lg'(x^L_T)\theta_T+\int_0^T\theta_tdM_t-\int_0^T\gamma^L_tg'(x_t)\theta_tdt
\ee
Since $M$ is a square integrable martingale, we obtain $E\left[\int_0^T\theta_tdM_t\right]=0$. Using convexity of $g$ once more, we obtain
\be
g(x^L_t)-g(x_t)\le g'(x^L_t)(x^L_t-x_t).
\ee
This implies optimality of $x^L$:
\be
E\left[\int_0^T\eta_t(g(\dot x^L_t)-g(\dot x_t))dt\right]&\le& -E\left[Lg'(x^L_T)\theta_T+\int_0^T\gamma^L_tg'(x_t)\theta_tdt\right]\\
&\le& -E\left[L(g(x^L_T)-g(x_T))+\int_0^T\gamma_t^L(g(x^L_t)-g(x_t))dt\right].
\ee
 It remains to verify the identity $v^L=Y^L_0$. To this end we apply the integration by parts formula to the process $Y(x^L)^p$ to obtain
\be
d(Y(x^L)^p)_t=-\left(\frac {(Y_t^L)^q}{\eta_t^{q-1}}(x_t^L)^p+\gamma^L_t(x_t^L)^p\right)dt+(x_t^L)^pZ^L_tdW_t.
\ee
Moreover we have
\be
\lvert \dot x^L_t \rvert^p=\left(\left(\frac {Y_t^L}{\eta_t}\right)^{q-1}x^L_t\right)^p=\left(\frac{Y^L_t}{\eta_t}\right)^q(x_t^L)^p.
\ee
Thus we obtain
\be
Y_0^L=E\left[\int_0^T\eta_t|\dot x^L_t|^p+\gamma^L_t|x^L_t|^pdt\right]=J^L(x^L)=v^L.
\ee
\end{proof}

\subsection{The constrained case}\label{constrained}

We now turn to the constrained case and prove Theorem \ref{sec main}. For the reader's convenience we restate the theorem here.
\begin{theo}\label{opt_strat_constr}
 Let $(Y,Z)$ be the minimal solution to \eqref{BSDE1} with singular terminal condition $Y_T=\infty$ from Theorem \ref{existence_singular}. Then $v = Y_0 |\xi|^p$; moreover the control $x_t=\xi \exp\left(-\int_0^t\left(\frac{Y_s}{\eta_s}\right)^{q-1} ds\right)$ belongs to $\mathcal A_0$ and is optimal in \eqref{min_problem}.
\end{theo}
\begin{proof}
To simplify notation assume that $\xi = 1$. As in the proof of Proposition \ref{penverification} we introduce $M_t=pY_tx_t^{p-1}+p\int_0^t\gamma_sx_s^{p-1}dt$. Performing integration by parts yields
\be
dM_t=x_t^{p-1}Z_tdW_t.
\ee
Hence, $M$ is a nonnegative local martingale on $[0,T)$ and in particular a nonnegative super-martingale. Thus it converges almost surely in $\IR$ as $t\nearrow T$. Since $Y$ satisfies the terminal condition $\liminf_{t\nearrow T}Y_t=\infty$ we have that
\be
0 \le x_t=\left(\frac{M_t-p\int_0^t\gamma_sx_s^{p-1} ds}{pY_t}\right)^{q-1} \le \left(\frac{M_t}{pY_t}\right)^{q-1}\to 0
\ee
 a.s.\ for $t\nearrow T$. It follows that $x\in \mathcal A_0$.

Next we apply the integration by parts formula to the process $Yx^p$ to obtain
\be
d(Yx^p)_t=-\left( \eta_t|\dot x_t|^p+\gamma_tx_t^p\right)dt+x_t^pZ_tdW_t.
\ee
Since $Z\in \mathcal M^2(0,t)$ and $|x_t|\le 1$ we can deduce for $t<T$
\be
Y_0=E\left[\int_0^t \big(\eta_s\lvert \dot x_s\rvert ^p +\gamma_sx_s^p\big)ds\right]+E\left[Y_tx_t^p\right]\ge E\left[\int_0^t\big(\eta_s\lvert \dot x_s\rvert ^p+\gamma_sx_s^p \big)ds\right].
\ee
Taking the limit $t\nearrow T$ and appealing to monotone convergence theorem yields
\ben\label{upboundY_0}
Y_0\ge E\left[\int_0^T\big(\eta_s\lvert \dot x_s\rvert ^p +\gamma_sx_s^p \big)ds\right] = J(x).
\een
Next, note that for every $\overline x \in \mathcal A_0$ we have $J(\overline x)\ge J^L(\overline x)$. This implies $v\ge v^L$ for every $L>0$. By Proposition \ref{penverification} we have $Y_0^L=v^L$. Minimality of $Y$ implies $Y_0=\lim_{L\nearrow \infty}Y^L_0=\lim_{L\nearrow \infty}v^L\le v$.  Consequently we obtain with Equation \eqref{upboundY_0}
\be
Y_0\ge J(x) \ge v \ge Y_0
\ee
 and thus optimality of $x$.
\end{proof}
\begin{remark}\label{general_inistate}
The solution $Y$ from Theorem \ref{existence_singular} does not only lead to optimal controls in the case where the liquidation period begins at time $t=0$  and the initial position position is equal to $x_0=1$ but also for general initial states.
  Let $x\in \mathcal A_0$ denote the optimal control from Theorem \ref{opt_strat_constr}. For a general initial position $\xi\in \IR$ the homogeneity of $z\mapsto |z|^p$ implies that the process $t\mapsto \xi x_t$ minimizes the functional $E\left[\int_0^T \big(\eta_t|\dot {\tilde x}_t|^p+\gamma_t|{\tilde x}_t|^p\big) dt\right]$ over all progressively measurable processes $\tilde x$ with absolutely continuous paths starting in $x_0$ and ending in $0$. The value of this minimization problem is then given by $Y_0|x_0|^p$.
If liquidation starts at an arbitrary time $t<T$ the minimization problem reads
\be
V_t=\inf E\left[\int_t^T(\eta_s\lvert \dot {\tilde x}_s \rvert^p +\gamma_s|{\tilde x}_s|^p\big)ds\bigg|\mathcal F_t\right],
\ee
where the infimum is taken over all progressively measurable processes $\tilde x$ starting in a $\mathcal F_t$-measurable random variable $\xi$ and ending in $0$. In this case the optimal control is given by
\be
x_s=\xi \exp\left(-\int_t^s\left(\frac{Y_r}{\eta_r}\right)^{q-1} dr\right)
\ee
and the value is equal to $V_t=Y_t|\xi|^p$.
\end{remark}


In the next proposition we state an integrability condition that allows to identify the minimal solution of \eqref{BSDE1}.
\begin{propo}\label{suff_cond_minimality}
 Let $(Y,Z)$ be a nonnegative solution of \eqref{BSDE1} with singular terminal condition $Y_T=\infty$. Let $x_t=\exp\left(-\int_0^t\left(\frac{Y_s}{\eta_s}\right)^{q-1} ds\right)$ denote the associated position path and assume that $x^{p-1}Z\in \mathcal M^2(0,T)$. Then $Y$ is the minimal solution of \eqref{BSDE1}.
\end{propo}
\begin{proof}
Let $Y^{\min}$ denote the minimal solution of \eqref{BSDE1}. Without loss of generality we only consider the point in time $t=0$ and show that $Y_0=Y_0^{\min}$. For general $t<T$ we refer to Remark \ref{general_inistate} which shows that $Y_t^{\min}$ is the value of the liquidation problem starting in time $t$. We proceed as in the proof of Theorem \ref{opt_strat_constr}. Let $M_t=pY_tx_t^{p-1}+p\int_0^t\gamma_sx_s^{p-1}dt$. Then we obtain by integration by parts
\be
dM_t=x_t^{p-1}Z_tdW_t.
\ee
Hence, $M$ is a nonnegative true martingale with $E[M_T^2]<\infty$ and converges a.s.\ in $\IR$ as $t\nearrow T$. Since $Y$ satisfies the terminal condition $\liminf_{t\nearrow T}Y_t=\infty$ we have that $x_t\to 0$ as $t\nearrow T$. Consequently, $x\in \mathcal A_0$ and Lemma \ref{suff_cond} implies optimality of $x$. Again an application of the integration by parts formula yields
\be
d(Yx^p)_t=(\eta_t|\dot x_t|^{p}+\gamma_tx_t^p)dt+x_t^pZ_tdW_t.
\ee
By assumption the process $t\mapsto \int_0^tx_t^pZ_tdW_t$ is a true martingale. Moreover we have $\lim_{t\nearrow T}Y_tx_t^p=0$ and hence Theorem \ref{opt_strat_constr} implies $Y_0=J(x)=v=Y_0^{\min}$.
\end{proof}

\section{Processes with uncorrelated multiplicative increments}\label{ind_incr_section}

In this section we study the special case of the control problem \eqref{min_problem} where $\gamma =0$ and $\eta$ has uncorrelated multiplicative increments. We first give a rigorous definition of what the latter means.

We say that a positive, progressively measurable process $\eta$ has uncorrelated multiplicative increments if $E\left[\frac{\eta_t}{\eta_s}|\mathcal F_s\right]=E\left[\frac{\eta_t}{\eta_s}\right]$ for all $s\le t<T$. We show that it is precisely this class of processes which leads to {\it deterministic} optimal controls for the minimization problem \eqref{min_problem} (with $\gamma = 0$). Moreover we show that if $\eta$ is a martingale, then it is optimal to close the position at a constant rate.

Observe that any process $\eta$ where $\frac {\eta_t}{\eta_s}$ is independent of $\mathcal F_s$ for $s\le t<T$ has uncorrelated multiplicative increments. The converse does not hold true.

In the next lemma we give an equivalent characterization of processes with uncorrelated multiplicative increments.
\begin{lemma}\label{charac_umi}
A positive, progressively measurable process $\eta$ has uncorrelated multiplicative increments if and only if the process $\left(\frac{\eta_t}{E[\eta_t]}\right)_{t<T}$ is a martingale. Any such process satisfies $E\left[\frac{\eta_t}{\eta_s}\right]=\frac{E[\eta_t]}{E[\eta_s]}$ for all $s\le t<T$.
\end{lemma}
\begin{proof}
 Let $\eta$ have uncorrelated multiplicative increments. We first show that for $s\le t<T$ any such $\eta$ satisfies $E\left[\frac{\eta_t}{\eta_s}\right]=\frac{E[\eta_t]}{E[\eta_s]}$. Indeed, we have
\be
E[\eta_t]=E\left[\eta_sE\left[\frac{\eta_t}{\eta_s}\bigg|\mathcal F_s\right]\right]=E[\eta_s]E\left[\frac{\eta_t}{\eta_s}\right].
\ee
Next let $M_t=\frac{\eta_t}{E[\eta_t]}$ for $t<T$. For $s\le t<T$ the process $M$ satisfies
\be
E[M_t|\mathcal F_s]=\frac 1{E[\eta_t]}E[\eta_t|\mathcal F_s]=\frac 1{E[\eta_t]}E\left[\frac{\eta_t}{\eta_s}\eta_s\bigg|\mathcal F_s\right]=\frac{\eta_s}{E[\eta_t]}E\left[\frac{\eta_t}{\eta_s}\right]=M_s.
\ee
For the converse direction, let  $M_t=\frac{\eta_t}{E[\eta_t]}$ be a martingale for $t<T$. Then we have for $s\le t<T$
\be
E[\eta_t|\mathcal F_s]=E[\eta_t]E[M_t|\mathcal F_s]=E[\eta_t]M_s=\frac{E[\eta_t]}{E[\eta_s]}\eta_s.
\ee
Thus the random variable $E\left[\frac{\eta_t}{\eta_s}|\mathcal F_s\right]$ is deterministic, which implies $E\left[\frac{\eta_t}{\eta_s}|\mathcal F_s\right]=E\left[\frac{\eta_t}{\eta_s}\right]$.
\end{proof}
Lemma \ref{charac_umi} implies that any positive martingale has uncorrelated multiplicative increments. Further examples are provided by the following class of diffusions.
\begin{ex}
 Let $\eta$ be a diffusion with linear drift, i.e.\ $\eta$ solves
\be
d\eta_t=\mu(t)\eta_tdt+\sigma(t,\eta_t)dW_t,
\ee
where the drift $\mu$ is a deterministic function of time and the stochastic volatility $\sigma:[0,T]\times \IR\times \Omega\to \IR_+$ is such that $t\mapsto\sigma(t,\eta_t)\in \mathcal M^2(0,T)$. Then the process $\eta  \exp(-\int_0^\cdot\mu(r)dr)$ is a martingale, and hence we have $E\left[\eta_t|\mathcal F_s\right]=\eta_s \exp(\int_s^t\mu(r)dr)$.
This implies that the random variable $E\left[\frac{\eta_t}{\eta_s}|\mathcal F_s\right]$ is deterministic. Therefore $\eta$ has uncorrelated multiplicative increments.
\end{ex}
We first show that if the optimal control from Theorem \ref{opt_strat_constr} is deterministic, then the process $\eta$ has necessarily uncorrelated multiplicative increments.
\begin{propo}\label{nec_umi}
Let $\eta$ be positive, progressively measurable and such that $\eta \in \mathcal M^2(0,T)$, $1/\eta^{q-1} \in \mathcal M^1(0,T)$. Assume that the optimal control $x\in \mathcal A_0$ from Theorem \ref{opt_strat_constr} is deterministic. Then $\eta$ has uncorrelated multiplicative increments.
\end{propo}
\begin{proof}
 The optimal control from Theorem \ref{opt_strat_constr} satisfies $\dot x_t=-\left(\frac{Y_t}{\eta_t}\right)^{q-1}x_t$ where $Y$ is the minimal solution of \eqref{BSDE1} with singular terminal condition $Y_T=\infty$. Since $x$ is deterministic it follows that the nonnegative process $\alpha_t= \left(\frac{Y_t}{\eta_t}\right)^{q-1}$ is deterministic as well. Furthermore $Y$ satisfies the linear BSDE
\be
dY_t=(p-1) \alpha_tY_tdt+Z_tdW_t
\ee
and hence Lemma \ref{LBSDE} implies for $s\le t<T$
\be
\alpha_s^{p-1}\eta_s=Y_s=E\left[Y_te^{-\int_s^t(p-1) \alpha_r dr}|\mathcal F_s\right]=\alpha_t^{p-1}e^{-\int_s^t (p-1) \alpha_r dr}E\left[\eta_t|\mathcal F_s\right].
\ee
Consequently, the random variable $E\left[\frac {\eta_t}{\eta_s}\big| \mathcal F_s\right]$ is deterministic for all $s\le t<T$ and hence $\eta$ has uncorrelated multiplicative increments.
\end{proof}

We next show that the converse of Proposition \ref{nec_umi} holds true as well: If $\eta$ has uncorrelated multiplicative increments, then there exists an deterministic optimal control for \eqref{min_problem}.
\begin{propo}\label{ind_incr}
Assume that $\eta$ has uncorrelated multiplicative increments and satisfies the integrability assumptions $\bf{(I1)}$ and $\eta_T\in L^2(\Omega)$. Then
\be
Y_t=\frac{1}{\left(\int_t^T\frac{1}{E[\eta_s|\mathcal F_t]^{q-1}}ds\right)^{p-1}}
\ee
is the minimal solution to \eqref{BSDE1} with singular terminal condition. The deterministic control
\be
x_t=\frac1{\int_0^T \frac{1}{E[\eta_s]^{q-1}}ds}\int_t^T\frac{1}{E[\eta_s]^{q-1}}ds
\ee
is optimal in \eqref{min_problem}. In particular the optimal control rate is inversely proportional to $E[\eta_t]^{q-1}$.
\end{propo}
\begin{proof}
 First note that we have by Jensen's inequality
\be
\int_t^T\frac{1}{E[\eta_s|\mathcal F_t]^{q-1}}ds\ge (T-t)^q\frac 1{\left(\int_t^TE[\eta_s|\mathcal F_t]ds\right)^{q-1}}.
\ee
This implies that $Y$ is bounded from above as follows
\ben\label{upboundY2}
Y_t\le \frac 1{(T-t)^p}E\left[\int_t^T\eta_sds|\mathcal F_t\right].
\een
Next we use the fact from Lemma \ref{charac_umi} that $E[\eta_s|\mathcal F_t]=\eta_tE\left[\frac{\eta_s}{\eta_t}\right]=\eta_t\frac{E[\eta_s]}{E[\eta_t]}$ for $s\ge t$ to rewrite $Y$ as
\be
Y_t=M_t\frac{1}{\left(\int_t^T\frac{1}{E[\eta_s]^{q-1}}ds\right)^{p-1}}
\ee
where the process $M$ denotes the martingale $M_t=\frac{\eta_t}{E[\eta_t]}$. Moreover, we have by assumption $E[M_T^2]=E[\eta_T^2]/E[\eta_T]^2<\infty$. Hence, $M$ is a square integrable martingale. Let $\phi\in\mathcal M^2(0,T)$ denote the integrand from its martingale representation. Then we obtain, by integration by parts,
\be
dY_t&=&(p-1)\frac{1}{E[\eta_t]^{q-1}}\frac{M_t}{\left(\int_t^T\frac 1{E[\eta_s]^{q-1}}ds\right)^{p}}dt+\frac{\phi_t}{\left(\int_t^T\frac{1}{E[\eta_s]^{q-1}}ds\right)^{p-1}}dW_t \\
&=&(p-1)\frac {Y^q_t}{\eta_t^{q-1}} d t+Z_t dW_t,
\ee
with
\ben\label{Zuncorincr}
Z_t=\frac{\phi_t}{\left(\int_t^T\frac{1}{E[\eta_s]^{q-1}}ds\right)^{p-1}}.
\een
Hence, we have $Z\in \mathcal M^2(0,t)$ for every $t<T$. An application of the Burkholder-Davis-Gundy inequality as in the proof of Theorem \ref{existence_singular} in combination with Inequality \eqref{upboundY2} yields $E[\sup_{0\le s\le t}Y_s^2]<\infty$ for all $t<T$. Hence, $(Y,Z)$ is a solution to \eqref{BSDE1} with singular terminal condition $Y_T=\infty$.

The associated path $x$ satisfies
\be
x_t=\exp\left(-\int_0^t\left(\frac{Y_s}{\eta_s}\right)^{q-1} ds\right)&=&\exp\left(-\int_0^t\frac{1}{E[\eta_s]^{q-1}\int_t^T\frac 1{E[\eta_r]^{q-1}}dr}ds\right)\\
&=&\frac1{\int_0^T \frac{1}{E[\eta_s]^{q-1}}ds}\int_t^T\frac{1}{E[\eta_s]^{q-1}}ds.
\ee
In particular it follows from \eqref{Zuncorincr} that $x^{p-1}Z\in \mathcal M^2(0,T)$ and hence Proposition \ref{suff_cond_minimality} yields that $Y$ is the minimal solution of \eqref{BSDE1}. Theorem \ref{opt_strat_constr} then implies optimality of $x$.
\end{proof}
If $\eta$ is monotone in expectation, then we obtain the following result about the path of the optimal control.
\begin{corollary}
 Let $\eta$ satisfy the assumptions of Proposition \ref{ind_incr}. If the mapping $t\mapsto E[\eta_t]$ is nondecreasing (nonincreasing), then the optimal control $x\in \mathcal A_0$ from Proposition \ref{ind_incr} is a convex (concave) function of time.
\end{corollary}
\begin{proof}
 The optimal control rate from Proposition \ref{ind_incr} is given by $\dot x_t=-\frac{1}{cE[\eta_t]^{q-1}}$ with $c=\int_0^T \frac{1}{E[\eta_s]^{q-1}}ds$. In particular $t\mapsto \dot x_t$ is nondecreasing (nonincreasing) if $t\mapsto E[\eta_t]$ is nondecreasing (nonincreasing).
\end{proof}
Proposition \ref{ind_incr} includes the case where $\eta$ is a martingale as a special case.
\begin{corollary}\label{martingale_case}
 Let $\eta$ be a positive martingale satisfying $1/\eta^{q-1}\in \mathcal M^1(0,T)$ and $\eta_T \in L^2(\Omega)$. Then $Y_t=\frac {\eta_t}{(T-t)^{p-1}}$ solves the BSDE \eqref{BSDE1} with singular terminal condition $Y_T=\infty$ and the control with constant control rate $x_t=1-\frac tT$ is optimal in \eqref{min_problem}.
\end{corollary}
\begin{proof}
 The process $\eta^2$ is a submartingale and hence $E[\eta_t^2]\le E[\eta_T^2]$ for all $t\le T$, which implies that $\eta \in \mathcal M^2(0,T)$. Moreover, Lemma \ref{charac_umi} yields that $\eta$ has uncorrelated multiplicative increments. Hence, all assumptions of Proposition \ref{ind_incr} are satisfied which yields the claim.
\end{proof}
Another special case of Proposition \ref{ind_incr} is the case where $\eta$ is a deterministic function of time.
\begin{corollary}\label{det_solution}
Assume that $\eta$ is deterministic and satisfies $1/\eta^{q-1} \in L^1([0,T])$, $\eta \in L^2([0,T])$ and $\eta_T<\infty$. Then
\be
Y_t=\left(\frac 1{\int_t^T\frac 1{\eta_s^{q-1}}ds}\right)^{p-1}
\ee
solves \eqref{BSDE1} with singular terminal condition $Y_T=\infty$ and the control
\ben\label{opt_strat_det}
x_t=\frac {\int_t^T\frac 1{\eta_s^{q-1}}ds}{\int_0^T\frac 1{\eta_s^{q-1}}ds}
\een
is optimal in \eqref{min_problem}.
\end{corollary}
\begin{remark}
 The results about the optimal control in Corollary \ref{martingale_case} and Corollary \ref{det_solution} hold also true under weaker assumptions on the process $\eta$. In the martingale case it suffices to assume that $\eta$ is a positive martingale with $E[\eta_T^2]<\infty$. Then Proposition \ref{suff_cond} directly implies that the control with constant rate is optimal. In the deterministic case it is straightforward to show that under the integrability condition $1/\eta^{q-1} \in L^1([0,T])$ the function $\eta|\dot x|^{p-1}$ is constant for the control $x$ from Equation \eqref{opt_strat_det}. Then again Proposition \ref{suff_cond} implies optimality of $x$.
\end{remark}
A particular example for a process with uncorrelated multiplicative increments is the geometric Brownian motion.
\begin{ex}
 Assume that $\eta$ evolves according to a geometric Brownian motion
\be
d\eta_t=\mu \eta_tdt+\sigma \eta_tdW_t
\ee
with drift $\mu\in \IR$, volatility $\sigma>0$ and initial value $\eta_0>0$. In this case
\be
\frac {\eta_t}{\eta_s}=e^{\left(\mu-\frac{\sigma^2}{2}\right)(t-s)+\sigma(W_t-W_s)}
\ee
for $s\le t\le T$ and hence $\eta$ has uncorrelated multiplicative increments. Moreover we have $E[\eta_t|\mathcal F_s]=\eta_se^{\mu(t-s)}$ and $\eta$ satisfies the integrability conditions $\eta \in \mathcal M^2(0,T)$, $E[\eta_T^2<\infty]$ and $\int_t^T\frac{1}{E[\eta_s]^{q-1}}ds<\infty$. In the case $\mu=0$ the price impact process $\eta$ is a martingale and Corollary \ref{martingale_case} yields that linear closure is optimal in \eqref{min_problem}. In the case $\mu \neq 0$ Proposition \ref{ind_incr} implies that a solution of \eqref{BSDE1} is given by
\be
Y_t=\mu(q-1)^{p-1}\frac{\eta_t}{\left(1-e^{-\mu(q-1)(T-t)}\right)^{p-1}}
\ee
and that the optimal control for \eqref{min_problem} satisfies
\be
x_t=\frac{e^{-\mu(q-1)t}-e^{-\mu(q-1)T}}{1-e^{-\mu(q-1)T}}.
\ee
\end{ex}

\section*{Appendix}
Here we provide a uniqueness result about linear BSDEs with a driver that is unbounded from below.
\begin{lemma}\label{LBSDE}
 Let $(\alpha_t)_{0\le t\le T}$ and $(\beta_t)_{0\le t\le T}$ be progressively measurable processes and $\xi$ a $\mathcal F_T$-measurable random variable. Assume that $\alpha$ is bounded from above. Any solution $(Y,Z)$ with $Z\in \mathcal M^2(0,T)$ to the linear BSDE
\be
dY_t=\left(\alpha_tY_t+\beta_t\right)dt+Z_tdW_t
\ee
with $Y_T=\xi$ admits the representation
\be
Y_t=E\left[\xi e^{\int_t^T\alpha_s ds}+\int_t^Te^{\int_t^s\alpha_u du}\beta_sds|\mathcal F_t\right].
\ee
\end{lemma}
\begin{proof}
Let $(Y,Z)$ be a solution. Set
\be
\varphi_t=Y_t e^{\int_0^t\alpha_s ds}+\int_0^te^{\int_0^s\alpha_u du}\beta_sds.
\ee
 Then by integration by parts we obtain
\be
d\varphi_t=e^{\int_0^t\alpha_s ds}Z_tdW_t.
\ee
Since $\alpha$ is bounded from above and $Z\in \mathcal M^2(0,T)$ the integrand belongs to $\mathcal M^2(0,T)$ as well. Therefore $\varphi$ is a martingale and consequently
\be
\varphi_t=E[\varphi_T|\mathcal F_t]=E\left[\xi e^{\int_0^T\alpha_s ds}+\int_0^Te^{\int_0^s\alpha_u du}\beta_sds|\mathcal F_t\right],
\ee
which yields the claim.
\end{proof}

\bibliographystyle{plain}
\bibliography{opti_liqui}

\end{document}